\documentclass[12pt]{article}
\usepackage[english]{babel}
\usepackage{subcaption}
\usepackage{cite}
\usepackage{amsmath,amssymb,amsfonts,amsthm}
\usepackage{algorithmic}
\usepackage{mathtools}
\usepackage{graphicx}
\usepackage{subfig}
\usepackage{textcomp}
\usepackage{xcolor}
\usepackage{comment}
\usepackage{array}
\usepackage{tabularx}
\usepackage{enumerate}

\theoremstyle{plain}
\newtheorem{theorem}{Theorem}[section]
\newtheorem{lemma}[theorem]{Lemma}

\newtheorem{corollary}[theorem]{Corollary}

\theoremstyle{definition}
\newtheorem{example}{Example}[section]
\newtheorem{definition}{Definition}[section]

\DeclarePairedDelimiter\ceil{\lceil}{\rceil}
\DeclarePairedDelimiter\floor{\lfloor}{\rfloor}

\def\BibTeX{{\rm B\kern-.05em{\sc i\kern-.025em b}\kern-.08em
    T\kern-.1667em\lower.7ex\hbox{E}\kereln-.125emX}}


\title{An upper bound on generalized cospectral mates of oriented graphs using skew-walk matrices}

\author{\small Muhammad Raza$^{{\rm a}}$\thanks{Corresponding author: mraza@itu.edu.pk}\quad\quad Obaid Ullah Ahmad$^{\rm b}$\quad\quad Mudassir Shabbir$^{\rm a}$\quad\quad Xenofon Koutsoukos$^{{\rm c}}$ \\
\small Waseem Abbas$^{{\rm b}}$
\\
{\footnotesize$^{\rm a}$Department of Computer Science, Lahore University of Management Sciences, Lahore, Pakistan}\\
{\footnotesize$^{\rm b}$Department of Systems Engineering, The University of Texas at Dallas, Richardson, TX}\\
{\footnotesize$^{\rm c}$Department of Computer Science, Vanderbilt University, Nashville, TN
}
}

\date{}

\begin{document}

\maketitle

\begin{abstract}
Let $D$ be an oriented graph with skew-adjacency matrix $S(D)$. Two oriented graphs $D$ and $C$ are said to share the same generalized skew spectrum if $S(D)$ and $S(C)$ have the same eigenvalues, and $J-S(D)$ and $J-S(C)$ also have the same eigenvalues, where $J$ is the all-ones matrix. Such graphs that are not isomorphic are generalized cospectral mates. We derive tight upper bounds on the number of generalized cospectral mates that an oriented graph can admit, based on arithmetic criteria involving the determinant of its skew-walk matrix. As a special case, we also provide a criterion for an oriented graph to be weakly determined by its generalized skew spectrum (WDGSS), that is, its only generalized cospectral mate is its transpose.
These criteria relate directly to the controllability of graphs, a fundamental concept in the control of networked systems, thereby connecting spectral characterization of graphs to graph controllability.

\end{abstract}

%========== Section: INTRODUCTION ===================

\section{Introduction}
\label{section:intro}

Let \( G = (V, E) \) be an undirected graph of order $n$. The \emph{adjacency matrix} \( A(G) \) of \( G \) is a symmetric \( n \times n \) matrix, where entry \( a_{ij} \) is $1$ if \( v_i \) and \( v_j \) are adjacent, and $0$ otherwise. 
% TODO: do we ever use degree matrix or Laplacian, Also this is duplicate use of D notation
% The \emph{degree matrix} \( D(G) \) is a diagonal matrix with \( D_{ii} \) equal to the degree of vertex \( v_i \), i.e., the number of vertices adjacent to \( v_i \). The \emph{Laplacian matrix} \( L(G) \) is defined as \( L(G) = D(G) - A(G) \).
For an undirected graph $G$ on $n$ vertices, the \emph{walk matrix} of $G$, denoted by $W(G)$, is the $n \times n$ matrix defined as $[e,\, A(G)e,\, \dots ,\, A(G)^{n-1}e]$, where $e$ is the all-ones column vector of dimension $n$. The graph $G$ is said to be \emph{controllable} if $W(G)$ is non-singular.

The \emph{adjacency spectrum} of an undirected graph $G$, denoted by $\lambda(A(G))$, is the multi-set of all eigenvalues of $A(G)$ (including multiplicities). Two graphs are said to be \emph{cospectral} if they have the same adjacency spectrum. For a real number $t$, two undirected graphs $G$ and $H$ are called \emph{$t$-cospectral} if the matrices $tJ - A(G)$ and $tJ - A(H)$ have the same spectrum, where $J$ is the matrix of all ones. We call two undirected graphs $G$ and $H$ \emph{$\mathbb{R}$-cospectral} if they are $t$-cospectral for every $t \in \mathbb{R}$. A notable result by Johnson and Newman~\cite{johnson1980note} establishes that if two undirected graphs are $t$-cospectral for two distinct values of $t$, then they are $\mathbb{R}$-cospectral.

Two graphs $G$ and $H$ are said to be \emph{isomorphic}, denoted by $G \cong H$, if there exists a bijection $f$ between their vertex sets such that any two vertices $u$ and $v$ are adjacent in $G$ if and only if $f(u)$ and $f(v)$ are adjacent in $H$. Isomorphic graphs are structurally identical despite possibly having different vertex labelings.

An undirected graph \( G \) is said to be \emph{determined by the spectrum} (henceforth, DS) if $\lambda(A(G))$ is unique for $G$ up to isomorphism, i.e., $\lambda(A(G))=\lambda(A(H))\implies  G \cong H$. A long-standing open problem in spectral graph theory is to characterize which graphs possess this property. It is widely conjectured that \emph{almost all} undirected graphs are DS \cite{van2003graphs}. Nevertheless, determining whether a specific graph is DS remains a difficult problem. For more background, see the survey papers \cite{van2003graphs, van2009developments}. A commonly used approach to address this and related problems involves the notion of the {\em generalized spectrum}. Two undirected graphs \( G \) and \( H \) are said to share the same generalized spectrum if they are cospectral and their complements are also cospectral. 

A pair of non-isomorphic graphs sharing the same generalized spectrum are called \emph{generalized cospectral mates} of each other. A graph \( G \) is said to be \emph{determined by the generalized spectrum} (DGS, henceforth) if it has no generalized cospectral mate.

The DGS framework was introduced by Wang and Xu \cite{wang2006excluding}, and a large family of undirected graphs that are DGS was identified in a subsequent work, \cite{wang2017simple}, using properties of the walk matrix \( W(G) \). More recently, Qiu et al. \cite{qiu2022new} extended these results to {\em almost} controllable undirected graphs. Later, Wang \cite{wang2023graphs} introduced a family of graphs having at most one generalized cospectral mate. All of these results leverage control-theoretic measures to study the spectral characterization of undirected graphs. For further results and background, see \cite{wang2023improved,wang2024haemers}.

The concept of spectral determination naturally extends to oriented graphs. An \emph{oriented graph} \(D\) is a directed graph (digraph) derived from an undirected graph \( G \) by assigning a direction to each edge of \( G \) based on a given orientation. The \emph{skew-adjacency matrix} of an oriented graph \( D \) is a variant of adjacency matrix that is crucial in the present context. Introduced by Tutte~\cite{tutte1947factorization},  it is the \( n \times n \) matrix \( S(D)=  (s_{ij}) \), where
\begin{equation}
    s_{ij}=
    \begin{cases}
         1 & \text{if $(v_i, v_j)$ is an arc;} \\
         -1 & \text{if $(v_j, v_i)$ is an arc;} \\
         0 & \text{otherwise.}
    \end{cases}
\end{equation}

Let \(W(D) = [e,\, S(D)e,\, \dots,\, S(D)^{n-1} e ]\) be the \emph{skew-walk matrix} of \(D\). We call the oriented graph \(D\), \emph{controllable}, if \(W(D)\) is nonsingular. The skew spectrum of \(D\), denoted by \(\lambda(S(D))\), is the multiset of all eigenvalues (with multiplicities) of its skew-adjacency matrix \(S(D)\).

% TODO: confirm the following.
% Correct up till the comment
We say two oriented graphs $D$ and $C$ share the same \emph{generalized skew spectrum} if $\lambda(S(D))$ equals $\lambda(S(C))$ and $\lambda(J - S(D))$ equals $\lambda(J - S(C))$, where $J$ is the all-ones matrix. By Johnson and Newman~\cite{johnson1980note}, this condition implies that \(D\) and \(C\) are \(\mathbb{R}\)-cospectral. Non-isomorphic graphs that share the same generalized skew spectrum are called \emph{generalized cospectral mates}.

Naturally, an oriented graph \(D\) is \emph{determined by the generalized skew spectrum} (or DGSS) if it has no generalized cospectral mate. For an oriented graph $D$, we define a \emph{transpose graph}, $D^{\mathrm{T}}$, by reversing the direction of each edge in $D$, i.e., $(u,v)$ is an edge in $D^{\mathrm{T}}$ if and only if $(v,u)$ is an edge in $D$. An oriented graph \( D \) is said to be \emph{self-transpose} if it is isomorphic to its transpose, \( D^{\mathrm{T}} \).

% The first sentence isn't accurate. Wang actually studied a subfamily of self transpose oriented graphs. I think my ord version was better 
Recently, Qiu et al.~\cite{qiu2021oriented} extended spectral characterization methods to oriented graphs and introduced a family of self-transpose oriented graphs that are DGSS. An alternative proof for their work was given by Li et al.~\cite{li2023smith}. Later, Chao et al.~\cite{chao2025} extended this work to a larger family of oriented graphs. Despite these advances, work on the spectral characterization of oriented graphs remains limited and focused mostly on self-transpose oriented graphs.

Building on these studies, note that an oriented graph \(D\) and its transpose \(D^{\mathrm{T}}\) always share the same generalized skew spectrum. Consequently, if \(D\) is not self-transpose, it cannot be DGSS since \(D^{\mathrm{T}}\) will act as its generalized cospectral mate. To broaden the scope of spectral characterization, we therefore define an oriented graph \(D\) to be \emph{weakly determined by the generalized skew spectrum} (WDGSS) if its only generalized cospectral mate is its transpose \(D^{\mathrm{T}}\).

% TODO: I am not happy about the text below till end of theorem statement but I will talk to
% Waseem about it.
In this work, we move beyond prior approaches that largely focus on determining the existence of a generalized cospectral mate for an undirected or oriented graph. Instead, we find an upper bound on the number of generalized cospectral mates of an oriented graph, linking this bound to the distinct odd prime factors of the determinant of its skew-walk matrix. This result is especially significant for graphs that are WDGSS, and we provide a criterion for identifying such graphs.

Define \(\mathcal{F}_n\) as the set of oriented graphs \(D\) of order \(n\) 
for which \(2^{-\lfloor \frac{n}{2} \rfloor} \det W(D)\) (which is always an integer) is an odd 
square-free integer. Our main result is stated in the following theorem:

\begin{theorem}
    \label{theorem:main}
    Let \( D \in \mathcal{F}_n \), then \( D \) has at most \( 2^k - 1 \) generalized cospectral 
    mates, where \( k \) is the number of distinct odd prime factors 
    of \( \det {W(D)} \).
\end{theorem}

The rest of the paper is organized as follows. In Section \ref{section:prelim}, we present definitions and preliminary results. In Section \ref{section:proof}, we introduce theorems and lemmas leading to proof of Theorem \ref{theorem:main}. In Section \ref{section:examples}, we illustrate examples for our results. The relationship between controllability and spectral characterization of graphs is discussed in Section \ref{section:controls}, while conclusions and future directions are discussed in Section \ref{sec:conclusion}. The discussion focuses on oriented graphs only through Section \ref{section:examples}.

% Additional discussions are presented in Sections \ref{section:controls} and \ref{section:conclusion}.

\section{Preliminaries}
\label{section:prelim}

In this section, we present key results from the literature that will be utilized in proving Theorem~\ref{theorem:main} in Section~\ref{section:proof}. We also outline our primary approach to establishing an upper bound on the number of cospectral mates, which is partially inspired by the methods in \cite{qiu2021oriented, wang2023graphs}. 

The following theorem offers a simple characterization of two oriented
graphs that share the same generalized skew-spectrum:

% \textcolor{blue}{Why is controllable important in the following theorem}

\begin{theorem}
    \label{theorem:q}
    \cite{li2023smith} Let \( D \) be a controllable oriented graph. Then there exists an oriented graph \( C \) such that \( D \) and \( C \) have the same generalized skew-spectrum if and only if there exists a unique rational orthogonal matrix \( Q \) such that:
    \begin{equation}
        \label{eq:q}
        Q^{\mathrm{T}} S(D) Q = S(C), \quad Qe = e,
    \end{equation}
    where \( S(D) \) and \( S(C) \) are the skew-adjacency matrices of \( D \) and \( C \), respectively, and \( e \) is the all-ones vector.
\end{theorem}

Let $\Gamma(D)$ denote the set of all rational orthogonal matrices defined by
\begin{equation*}
    \Gamma(D) = \{Q \in O_n(\mathbb{Q}) \mid Q^{\mathrm{T}} S(D) Q = S(C) \text{ for oriented graph } C \text{ and } Qe = e\},
\end{equation*}
where \(O_n(\mathbb{Q})\) is the set of all orthogonal matrices with rational entries.

\begin{definition}
    Let \(Q\) be an orthogonal matrix with rational entries. The \emph{level} of \(Q\), denoted by \(\ell(Q)\) (or simply \(\ell\)), is the smallest positive integer \(x\) such that \(xQ\) is an integral matrix.
\end{definition}

Note that for any $Q \in \Gamma(D)$, we have $\ell(Q)=1$ if and only if $Q$ is a permutation matrix. Furthermore, the next lemma shows that if two matrices in $\Gamma(D)$ differ only by a permutation matrix, then their corresponding oriented graphs are isomorphic.

\begin{lemma}
\label{lemma:permutation}
Let $D \in \mathcal{F}_n$ and $Q_1, Q_2 \in \Gamma(D)$. Let $B$ and $C$ be the oriented graphs having skew-adjacency matrices:
\[
S(B)=Q_1^{\mathrm{T}} S(D) Q_1
\quad\text{and}\quad
S(C)=Q_2^{\mathrm{T}} S(D) Q_2.
\]
If $Q_2 = Q_1 P$, where $P$ is a permutation matrix, then $B$ and $C$ are isomorphic.
\end{lemma}

\begin{proof}
    We can express $S(D)$ in terms of $S(C)$ as:
    \[
    S(D)=Q_2 S(C) Q_2^{\mathrm{T}}
    \]
    Therefore, we can express $S(B)$ in terms of $S(C)$ as:
    \begin{equation}
    \label{eq:sksh}
    S(B) = Q_1^{\mathrm{T}} Q_2 S(C) Q_2^{\mathrm{T}} Q_1
    \end{equation}
    Since $Q_2 = Q_1 P$, we substitute to obtain:
    \begin{equation}
    \label{eq:skshperm}
    S(B) = Q_1^{\mathrm{T}} Q_1 P S(C) P^{\mathrm{T}} Q_1^{\mathrm{T}} Q_1 = P S(C) P^{\mathrm{T}}
    \end{equation}

Thus, Equation \eqref{eq:skshperm} shows that the skew-adjacency matrix $S(B)$ is obtained from $S(C)$ by a permutation similarity transformation. Hence, the oriented graphs $B$ and $C$ are isomorphic.
\end{proof}

Our approach is to show that if two matrices, \( Q_1 \) and \( Q_2 \), in \( \Gamma(D) \) have the same level \( \ell \), then \( Q_2 \) can be obtained from \( Q_1 \) by simply permuting its columns. By leveraging the preceding lemma, we can then characterize the sets of generalized cospectral mates and establish an upper bound on their number.

Smith Normal Form (SNF) is an important tool for working with integral matrices. An integral matrix \(V\) of order \(n\) is said to be \emph{unimodular} if \(\det{V}=\pm 1\). We denote by \( \mathbb{F}_p \) the finite field with \(p\) elements, and by \( \operatorname{rank}_p(M) \) the rank of an integral matrix \(M\) considered over \( \mathbb{F}_p \). 
% The following theorem is well known.

\begin{theorem}
    Let \( M \) be a full-rank integral matrix. There exist unimodular matrices \( V_1 \) and \( V_2 \) such that \( M = V_1 N V_2 \), where \( N = \text{diag}(d_1, d_2, \ldots, d_n) \) is the SNF with positive integers $d_i$ such that \( d_i \mid d_{i+1} \) for \( i = 1, 2, \ldots, n-1 \).
\end{theorem}

For an integral matrix $M$, let $d_i(M)$ represent the $i^\text{th}$ invariant factor of $M$. Note that the determinant of $M$ can be expressed as:
\begin{equation}
    \det{M} = \pm \prod_{i=1}^{n} d_i(M).
\end{equation}

It is easy to observe that for a prime number $p$, if $p \mid \det{M}$ and $p^2 \nmid \det{M}$, then $p \mid d_n(M)$ and $p \nmid d_i(M)$ for all $i \ne n$. Consequently, we have $\operatorname{rank}_p(N) = n - 1$ as only $d_n(M)$ is divisible by $p$. Moreover, since $V_1$ and $V_2$ are unimodular, it follows that $\operatorname{rank}_p(M) = \operatorname{rank}_p(N) = n - 1$. The following lemma plays a significant role in building a relationship between the level of the matrix $Q$ and the invariant factor $d_n$.
% defining a key property of the level of the  the proof of Theorem \ref{theorem:main}.

\begin{lemma}
    \label{lemma:dnq}
    \cite{qiu2023smith}
    Let $X$ and $Y$ be two non-singular integral matrices such that $QX = Y$, where $Q$ is a rational orthogonal matrix. Then $\ell(Q) \mid \operatorname{gcd}(d_n(X), d_n(Y))$ (where $\text{gcd}$ denotes the greatest common divisor).
\end{lemma}

The following two lemmas were originally proven by Qiu et al.~\cite{qiu2021oriented} for self-transpose oriented graphs in $\mathcal{F}_n$. However, the same reasoning is equally applicable to all oriented graphs that belong to $\mathcal{F}_n$.

\begin{lemma}
\label{lemma:snf-structure}
\cite{qiu2021oriented}
Let $D \in \mathcal{F}_n$, then $\operatorname{rank}_2(W(D)) = \ceil{\frac{n}{2}}$ and the SNF of $W(D)^{\mathrm{T}}$ is $N={\rm diag}(\underbrace{1,1,\ldots,1}_{\lceil\frac{n}2\rceil},\underbrace{2,2,\ldots,2,2b}_{\lfloor \frac{n}2 \rfloor})$, where b is an odd square-free integer.
\end{lemma}

\begin{lemma}
    \label{lemma:2-out}
    \cite{qiu2021oriented} Let \( D \in \mathcal{F}_n \) be an oriented graph of order \( n \), and let \( Q \in \Gamma(D) \) with level $\ell(Q)$. Then \( \ell(Q) \) is odd.
\end{lemma}

\section{Proof of Theorem \ref{theorem:main}}
\label{section:proof}

In this section, we establish key properties of rational orthogonal matrices. 
% in $\Gamma(D)$. 
The ensuing lemmas lay the foundation for the proof of Theorem~\ref{theorem:main}. We also provide a criterion for identifying a class of WDGSS graphs as a special case.

\begin{lemma}
\label{lemma:level-mod}
Let \( Q \) be a rational matrix, and let $x$ and $k$ be positive integers such that $k \mid x$. If \( xQ \) is an integral matrix and satisfies \( xQ \equiv 0 \pmod{k} \), then \( \ell(Q) \mid \frac{x}{k} \).
\end{lemma}

\begin{proof}
Since \( xQ \equiv 0 \pmod{k} \), every entry of \( xQ \) is divisible by \( k \). Therefore, \( \frac{x}{k}Q \) is also an integral matrix. By the definition of the level \( \ell(Q) \), the smallest integer such that \( \ell(Q)Q \) is an integral matrix, it follows that \( \ell(Q) \) must divide \( \frac{x}{k} \). This completes the proof.
\end{proof}

\begin{lemma}
\label{lemma:q}
Let \( D \) be a controllable oriented graph, and \( C\) be an oriented graph that has the same generalized skew-spectrum as \( D \). The rational orthogonal matrix $Q$ satisfying equation (\ref{eq:q}) can be computed as:
\begin{equation}
    Q = W(D) W(C)^{-1}.
\end{equation}
\end{lemma}
\begin{proof}
Using the facts that $Q$ is orthogonal and $Qe=e$, from equation (\ref{eq:q}), we have \( Q^{\mathrm{T}} S(D)^k e = S(C)^k e \) where $k$ is a positive integer. As a result, the skew-walk matrix \( W(C) \) can be written as:
\begin{equation}
\label{eq:w}
W(C) = Q^{\mathrm{T}} W(D).
\end{equation}
Thus, the matrix \( Q \) can be computed as $Q = W(D) W(C)^{-1}$.
\end{proof}

\begin{lemma}
\label{lemma:level-dn}
Let \(D \in \mathcal{F}_n\), \(Q \in \Gamma(D)\), and \(C\) be the oriented graph satisfying \(S(C)=Q^{\mathrm{T}} S(D)\,Q\). Then \(\ell(Q)\) divides both \(d_n(W(D))\) and \(d_n(W(C))\). Moreover, \(C\) also lies in \(\mathcal{F}_n\).
\end{lemma}

\begin{proof}
From equation (\ref{eq:w}), we have
\[
    W(C) = Q^{\mathrm{T}} W(D).
\]
Since \(Q\) is orthogonal with \(\det{Q^{\mathrm{T}}}=\pm1\), it follows that
\[
    \det{W(C)} = \det{Q^{\mathrm{T}}}\det{W(D)} = \pm\det{W(D)}.
\]
Thus, the quantity \(2^{-\lfloor \frac{n}{2} \rfloor}\det{W(C)}\) remains odd and square-free, implying that \(C\) belongs to \(\mathcal{F}_n\). Furthermore, by Lemma~\ref{lemma:snf-structure}, we have:
\[
    d_n(W(D)) = d_n(W(C)).
\]
Finally, applying Lemma~\ref{lemma:dnq} to the relation \(W(C)=Q^{\mathrm{T}} W(D)\) shows that \(\ell(Q)\) divides \(d_n(W(D))\) (and hence \(d_n(W(C))\)). This completes the proof.
\end{proof}

For convenience, we abbreviate the notation by writing
\[
S := S(D) \quad \text{and} \quad W := W(D),
\]
denoting the skew-adjacency matrix and the skew-walk matrix of \( D \), respectively, throughout the rest of this section. Note that for any matrix \( Q \in \Gamma(D) \), $d_n(W) Q$ is an integral matrix.
\begin{lemma}
    \label{lemma:rank}
    Let \( D \in \mathcal{F}_n \) and \( Q \in \Gamma(D) \). Define \( \bar{Q} = d_n(W) Q \). Then for every odd prime \( p \) that divides \( \ell(Q) \),
    \begin{enumerate}[(i)]
        \item \( \bar{Q}^{\mathrm{T}} \bar{Q} \equiv 0 \pmod{p^2} \),
        \item \( \operatorname{rank}_p(\bar{Q}) = 1 \).
    \end{enumerate}
\end{lemma}

\begin{proof}
By Lemma~\ref{lemma:level-dn}, we have that $\ell(Q) \mid d_n(W)$, so for any odd prime $p$ dividing $\ell(Q)$, it follows that $p \mid d_n(W)$. Let $u_i$ and $u_j$ denote the $i^\text{th}$ and $j^\text{th}$ columns of $\bar{Q}$, where $\bar{Q} = d_n(W)Q$. Since $Q$ is orthogonal, we obtain 
$$
u_i^{\mathrm{T}} u_j = 
\begin{cases}
d_n(W)^2, & \text{if } i = j,\\[1mm]
0, & \text{if } i \neq j.
\end{cases}
$$
Since $p^2 \mid d_n(W)^2$, it follows immediately that each inner product satisfies $u_i^{\mathrm{T}} u_j \equiv 0 \pmod{p^2}$, so that $\bar{Q}^{\mathrm{T}} \bar{Q} \equiv 0 \pmod{p^2}$.

Next, since $p$ is odd, by the assumptions of Theorem~\ref{theorem:main} we have $p^2 \nmid \det{W}$ and hence $p^2 \nmid d_n(W)$, it follows that $\operatorname{rank}_p(W)= n-1$. Moreover, from equation (\ref{eq:w}) we deduce that $W^{\mathrm{T}}Q$ is an integral matrix, which implies $W^{\mathrm{T}}\bar{Q} \equiv 0 \pmod{p}$. Since the null space of $W^\mathrm{T}$ over $\mathbb{F}_p$ is one-dimensional, we conclude that $\operatorname{rank}_p(\bar{Q}) \le 1$.

On the other hand, because $p \mid \ell(Q)$, by Lemma~\ref{lemma:level-mod}, $\bar{Q} \pmod{p}$ must contain some nonzero entries, so that $\operatorname{rank}_p(\bar{Q}) > 0$. Combining these observations, we obtain $\operatorname{rank}_p(\bar{Q}) = 1$ which completes the proof.
\end{proof}

We now examine the relationships between the various matrices in $\Gamma(D)$. The following corollary follows directly from the steps outlined in Lemma~\ref{lemma:rank}.

\begin{corollary}
\label{corollary:linear-basis}

Let $D\in\mathcal{F}_n$ and \( Q_1, Q_2 \in \Gamma(D) \). If an odd prime \( p \) divides both \( \ell(Q_1) \) and \( \ell(Q_2) \), then the column spaces of \( d_n(W) Q_1 \) and \( d_n(W) Q_2 \) over \( \mathbb{F}_p \) coincide and are one-dimensional.
\end{corollary}

% \textcolor{red}{The following Lemma is inspired by Lemma 7 of \cite{wang2021graphs}.}

\begin{lemma}
\label{lemma:uv}
Let \( u \) and \( v \) be two \( n \)-dimensional integral column vectors, and let \( p \) be an odd prime. If the following conditions hold:
\begin{enumerate}[(i)]
    \item \( u \not\equiv 0 \pmod{p} \) and \( v \not\equiv 0 \pmod{p} \),
    \item \( u \) and \( v \) are linearly dependent over \( \mathbb{F}_p \),
    \item \( u^{\mathrm{T}} u \equiv v^{\mathrm{T}} v \equiv 0 \pmod{p^2} \),
\end{enumerate}
then \( u^{\mathrm{T}} v \equiv 0 \pmod{p^2} \).
\end{lemma}

\begin{proof}
    If \( u = \pm v \), then \( u^{\mathrm{T}} v = \pm u^{\mathrm{T}} u \equiv 0 \pmod{p^2} \), so the result follows. Otherwise, since \( u \) and \( v \) are linearly dependent over \( \mathbb{F}_p \), there exist integers \( a \) and \( b \), not both zero in \( \mathbb{F}_p \), such that
    \begin{equation}
        \label{eq:uv1}
        au + bv \equiv 0 \pmod{p}.
    \end{equation}

    We first show that neither \( a \) nor \( b \) is congruent to zero modulo \( p \). Suppose, for example, that \( a \equiv 0 \pmod{p} \). Then, the equation becomes
    \[
    bv \equiv 0 \pmod{p}.
    \]
    Since \( b \not\equiv 0 \pmod{p} \), it follows that \( v \equiv 0 \pmod{p} \), contradicting our assumption about \( v \). A similar argument shows that \( b \not\equiv 0 \pmod{p} \). Hence,  neither \( a \) nor \( b \) is congruent to zero modulo \( p \).

    Now, taking the inner product of both sides of (\ref{eq:uv1}) with itself, we get
    \[
    (au + bv)^{\mathrm{T}} (au + bv) \equiv 0 \pmod{p^2}.
    \]
    Expanding this expression yields
    \[
    a^2 u^{\mathrm{T}} u + 2ab u^{\mathrm{T}} v + b^2 v^{\mathrm{T}} v \equiv 0 \pmod{p^2}.
    \]
    Since \( u^{\mathrm{T}} u \equiv v^{\mathrm{T}} v \equiv 0 \pmod{p^2} \), this simplifies to
    \begin{equation}
        \label{eq:uv2}
        2ab u^{\mathrm{T}} v \equiv 0 \pmod{p^2}.
    \end{equation}

    % \textcolor{blue}{Might have to mention that we assume $p \neq 2$ because of the 2ab thing}
    
    Given that \( 2ab \not\equiv 0 \pmod{p} \), equation (\ref{eq:uv2}) implies \( u^{\mathrm{T}} v \equiv 0 \pmod{p^2} \), completing the proof.
\end{proof}

% \textcolor{blue}{Check this Lemma for Example 1}    
% \textcolor{red}{maybe a more intuitive notation for O?} 
\begin{lemma}
\label{lemma:p-out}
Let $D\in\mathcal{F}_n$ and let \( Q_1, Q_2 \in \Gamma(D) \). If \( p \) is a common factor of \( \ell(Q_1) \) and \( \ell(Q_2) \), then \( p \nmid \ell(Q_1^{\mathrm{T}} Q_2) \).
\end{lemma}

\begin{proof}
Let \(\bar{Q}_1 = d_n(W)Q_1\) and \(\bar{Q}_2 = d_n(W)Q_2\). Define \(O = \bar{Q}_1^{\mathrm{T}} \bar{Q}_2\), where each entry \(o_{ij}\) is given by \(o_{ij} = u_i^{\mathrm{T}} v_j\). Here, \(u_i\) and \(v_j\) denote the \(i^\text{th}\) and \(j^\text{th}\) columns of \(\bar{Q}_1\) and \(\bar{Q}_2\), respectively. Note that \(Q_1^{\mathrm{T}} Q_2\) is a rational matrix, while \(\bar{Q}_1\), \(\bar{Q}_2\), and \(O\) are integral. We begin with proving that \(O \equiv 0 \pmod{p^2}\).

By Corollary~\ref{corollary:linear-basis}, each column \( u_i \) and \( v_j \) is linearly dependent over \( \mathbb{F}_p \). Using Lemma~\ref{lemma:rank}, we know \( u_i^{\mathrm{T}} u_i \equiv 0 \pmod{p^2} \) and \( v_j^{\mathrm{T}} v_j \equiv 0 \pmod{p^2} \). We consider the following cases:

\noindent \textbf{Case 1:} If \( u_i \not\equiv 0 \pmod{p} \) and \( v_j \not\equiv 0 \pmod{p} \), then by Lemma~\ref{lemma:uv}, we have \( o_{ij} = u_i^{\mathrm{T}} v_j \equiv 0 \pmod{p^2} \).  \vspace{1.5mm}

\noindent \textbf{Case 2:} If \( u_i \equiv 0 \pmod{p} \) and \( v_j \equiv 0 \pmod{p} \), then clearly \( u_i^{\mathrm{T}} v_j \equiv 0 \pmod{p^2} \). \vspace{1.5mm}

\noindent \textbf{Case 3:} If only one of \( u_i \) or \( v_j \) is congruent to zero modulo \( p \). Let's assume without loss of generality that \( u_i \equiv 0 \pmod{p} \) while \( v_j \not\equiv 0 \pmod{p} \). As $\operatorname{rank}_p(\bar{Q}_1) > 0$, there must be a column \( u_k \) of \( \bar{Q}_1 \) such that \( u_k \not\equiv 0 \pmod{p} \).

We have that \( u_k \) and \( v_j \) are linearly dependent over \( \mathbb{F}_p \). Thus, we can express \( v_j \) as:
\[
    v_j = c u_k + p \beta,
\]
where \( c \) is a nonzero integer and \( \beta \) is an integral vector. Multiplying both sides by $u_i^{\mathrm{T}}$, we get:
% Substituting this expression, we obtain:
\[
    u_i^{\mathrm{T}} v_j = c u_i^{\mathrm{T}} u_k + p u_i^{\mathrm{T}} \beta.
\]

By Lemma~\ref{lemma:rank}, \( c u_i^{\mathrm{T}} u_k \equiv 0 \pmod{p^2} \). Since \( u_i \equiv 0 \pmod{p} \), the term \( p u_i^{\mathrm{T}} \beta \) is also congruent to zero modulo \( p^2 \). Hence, \( o_{ij} = u_i^{\mathrm{T}} v_j \equiv 0 \pmod{p^2} \) in this case as well. The case where \( u_i \not\equiv 0 \pmod{p} \) and \( v_j \equiv 0 \pmod{p} \) can be proven similarly.

% \textcolor{blue}{Is $O \equiv 0$ (mod $p^3$)?}

Thus, we conclude that \( O = d_n(W)^2\, Q_1^{\mathrm{T}} Q_2 \equiv 0 \pmod{p^2} \). As \( Q_1^{\mathrm{T}} Q_2 \) is a rational matrix, by Lemma~\ref{lemma:level-mod}, it follows that \( \ell(Q_1^{\mathrm{T}} Q_2) \mid \frac{d_n(W)^2}{p^2} \). Since \( p^2 \nmid d_n(W) \), we also have \( p \nmid \frac{d_n(W)^2}{p^2} \). Hence, \( p \nmid \ell(Q_1^{\mathrm{T}} Q_2) \), and the proof is complete.
\end{proof}

% \textcolor{red}{Rephrase}
\begin{lemma}
\label{lemma:isomorphic}
Let $D\in\mathcal{F}_n$ and let $Q_1, Q_2 \in \Gamma(D)$. If $\ell(Q_1) = \ell(Q_2)$, then $Q_2 = Q_1 P$ where $P$ is a permutation matrix.
\end{lemma}

\begin{proof}
Let $B$ and $C$ be the oriented graphs with skew-adjacency matrices
\[
S(B)= Q_1^{\mathrm{T}}S(D)Q_1
\quad\text{and}\quad
S(C) = Q_2^{\mathrm{T}}S(D)Q_2.
\]
From Equation~\eqref{eq:sksh}, we can conclude that $Q_1^{\mathrm{T}} Q_2 \in \Gamma(B)$. Furthermore, by Lemma~\ref{lemma:level-dn}, both $B$ and $C$ belong to $\mathcal{F}_n$, therefore $\ell(Q_1^{\mathrm{T}} Q_2) \mid d_n(W(D))$.

As we obtain an integral matrix by multiplying \(\ell(Q_1)\ell(Q_2)\) with \(Q_1^{\mathrm{T}} Q_2\), \(\ell(Q_1^{\mathrm{T}} Q_2)\) must divide \(\ell(Q_1)\ell(Q_2)\). Moreover, since \(d_n(W(D))\) is square-free, \(\ell(Q_1^{\mathrm{T}} Q_2)\) is also square-free. Hence, if \(\ell(Q_1)=\ell(Q_2)\), then \(\ell(Q_1^{\mathrm{T}} Q_2)\mid\ell(Q_1)\). 

Additionally, by Lemma~\ref{lemma:2-out}, we know that $\ell(Q_1^{\mathrm{T}} Q_2)$ is odd. Now, if $p$ is any odd prime divisor of $\ell(Q_1)$, then by Lemma~\ref{lemma:p-out} we have \( p\nmid \ell(Q_1^{\mathrm{T}} Q_2) \). Thus, no odd prime divides $\ell(Q_1^{\mathrm{T}} Q_2)$, which forces \( \ell(Q_1^{\mathrm{T}} Q_2)=1 \).

This implies that $Q_1^{\mathrm{T}} Q_2$ is an integral orthogonal matrix with level one — that is, a permutation matrix $P$. Consequently, we obtain \( Q_2=Q_1P \) which completes the proof.
\end{proof}

With all the necessary tools in place, we now present the proof of Theorem \ref{theorem:main}.

\begin{proof}[Proof of Theorem \ref{theorem:main}]
Let \(B\) and \(C\) be two generalized cospectral mates of \(D\). Let \(Q_1, Q_2 \in \Gamma(D)\) such that
\[
S(B)= Q_1^{\mathrm{T}}S(D)Q_1
\quad\text{and}\quad
S(C) = Q_2^{\mathrm{T}}S(D)Q_2.
\]

By Lemma~\ref{lemma:isomorphic}, if \(\ell(Q_1)=\ell(Q_2)\) then \(Q_2=Q_1P\) for some permutation matrix \(P\); consequently, by Lemma~\ref{lemma:permutation}, the oriented graphs \(B\) and \(C\) are isomorphic. Thus, to obtain generalized cospectral mates, the corresponding matrices must have distinct levels.

From Lemma~\ref{lemma:2-out} and Lemma~\ref{lemma:level-dn}, every \(Q\in\Gamma(D)\) satisfies \(\ell(Q) \mid d_n(W(D))\) and \(\ell(Q)\) is odd. Therefore, the possible values of \(\ell(Q)\) are divisors of \(d_n(W(D))\) formed from the odd prime factors. Let \(k\) denote the number of distinct odd prime factors of \(\det{W(D)}\); then there are at most \(2^k\) distinct divisors.

Excluding the trivial case \(\ell(Q)=1\) (which corresponds to a permutation matrix yielding an isomorphic oriented graph), \(D\) can have at most \(2^k-1\) generalized cospectral mates. This completes the proof.
\end{proof}

Theorem~\ref{theorem:main} establishes an upper bound on the number of generalized cospectral mates that an oriented graph can possess. The following corollary then characterizes a special subclass of graphs that can be WDGSS.

\begin{corollary}
\label{corollary:wdgss}
Let \( D \) be an oriented graph of order \( n \). If \( D \) is not self-transpose and \( 2^{-\floor{\frac{n}{2}}} \det{W(D)} \) is an odd prime number, then \( D \) is WDGSS.
\end{corollary}

%========== Section: EXAMPLE ILLUSTRATIONS ===================
% \textcolor{red}{Diagrams are to be added.}
\section{Examples}
\label{section:examples}
In this section, we present examples of graphs satisfying the upper bound of Theorem~\ref{theorem:main}. We generated a dataset of all possible oriented graphs of order up to 7 and selected the following examples from it.

\begin{example}
Let $n=7$ and the skew-adjacency matrix of an oriented graph $D$ be as follows:

\begin{figure}[!t]
\centering
\begin{subfigure}[b]{0.49\linewidth}
% \hspace{-7.5mm}
\centering
\includegraphics[scale=1]{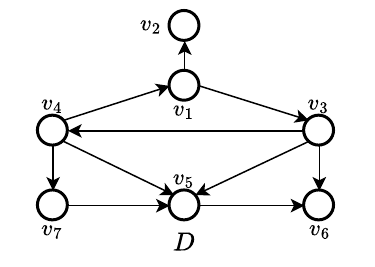}
\end{subfigure}
\begin{subfigure}[b]{0.49\linewidth}
% \hspace{-3mm}
\centering
\includegraphics[scale=1]{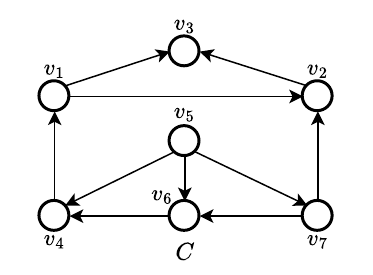}
\end{subfigure}
\caption{Two generalized cospectral mates.}
\label{fig:example_main_theorem}
\end{figure}

\begin{equation*}
S(D) = \begin{pmatrix}
  0 &  1 &  1 & -1 &  0 &  0 &  0 \\
 -1 &  0 &  0 &  0 &  0 &  0 &  0 \\
 -1 &  0 &  0 &  1 &  1 &  1 &  0 \\
  1 &  0 & -1 &  0 &  1 &  0 &  1 \\
  0 &  0 & -1 & -1 &  0 &  1 & -1 \\
  0 &  0 & -1 &  0 & -1 &  0 &  0 \\
  0 &  0 &  0 & -1 &  1 &  0 &  0
\end{pmatrix}
\end{equation*}

The determinant of the skew-walk matrix of $D$ is:
\[
\det{W(D)} = -14392 = (-1) \times 2^3 \times 7 \times 257.
\]
Hence, by Theorem~\ref{theorem:main}, $D$ can have at most three generalized cospectral mates, up to isomorphism. We find that there are exactly three generalized cospectral mates. We represent them by $D^{\mathrm{T}}$, $C$, and $C^{\mathrm{T}}$, where

\begin{equation*}
    S(C) =
    \begin{pmatrix}
    0 &  1 &  1 & -1 &  0 &  0 &  0 \\
    -1 &  0 &  1 &  0 &  0 &  0 & -1 \\
    -1 & -1 &  0 &  0 &  0 &  0 &  0 \\
    1 &  0 &  0 &  0 & -1 & -1 &  0 \\
    0 &  0 &  0 &  1 &  0 &  1 &  1 \\
    0 &  0 &  0 &  1 & -1 &  0 & -1 \\
    0 &  1 &  0 &  0 & -1 &  1 &  0
    \end{pmatrix}.
\end{equation*}

The two graphs $D$ and $C$ are shown in Figure \ref{fig:example_main_theorem}. Note that since $D$ and $C$ are not self-transpose, their transposes, $D^{\mathrm{T}}$ and $C^{\mathrm{T}}$ are also generalized cospectral mates of $D$. The skew-adjacency matrices of the transposes, $S(D^{\mathrm{T}})$ and $S(C^{\mathrm{T}})$, can be expressed as $S(D^{\mathrm{T}})=-S(D)$ and $S(C^{\mathrm{T}})=-S(C)$.

Let $Q_1, Q_2, Q_3 \in \Gamma(D)$ such that, $S(C) = Q_1^{\mathrm{T}} S(D) Q_1$, $S(D^{\mathrm{T}}) = Q_2^{\mathrm{T}} S(D) Q_2$ and $S(C^{\mathrm{T}}) = Q_3^{\mathrm{T}} S(D) Q_3$. The values of these matrices can be computed using Lemma~\ref{lemma:q}. We find that the levels are $\ell(Q_1)=7$, $\ell(Q_2)=7\times257=1799$ and $\ell(Q_3)=257$. These levels are distinct and correspond to the multiples of odd prime factors of $\det{W(D)}$ as predicted by Theorem \ref{theorem:main}. Hence, this is a tight example of upper bound proposed by Theorem~\ref{theorem:main}.

%Figure~\ref{fig:example_main_theorem}. 
\end{example}

\begin{example}
Let $n=6$ and the skew-adjacency matrix of an oriented graph $D$ be as follows:

% \begin{figure}[!t]
% \vspace{-20mm}
% \centering
% \includegraphics[width = \textwidth]{New_figs/Fig3.pdf}
% \vspace{-5mm}
% \caption{An example graph of WDGSS that satisfies the condition mentioned in Remark \ref{remark:wdgss}.}
% \label{fig:example_wdgss}
% \end{figure}

\begin{equation*}
S(D) = \begin{pmatrix}
  0 &  1 & -1 & -1 &  0 &  0 \\
 -1 &  0 &  0 &  0 &  0 &  0 \\
  1 &  0 &  0 & -1 & -1 &  0 \\
  1 &  0 &  1 &  0 & -1 & -1 \\
  0 &  0 &  1 &  1 &  0 &  0 \\
  0 &  0 &  0 &  1 &  0 &  0
\end{pmatrix}
\end{equation*}

The oriented graphs $D$ and $(D)^{\mathrm{T}}$ are shown in Figure \ref{fig:example_wdgss}. The determinant of the skew-walk matrix $W(D)$ turns out to be:
\[
\det{W(D)} = 1528 = 2^3 \times 191.
\]

\begin{figure}[!t]
\centering
\begin{subfigure}[b]{0.49\linewidth}
\hspace{-7.5mm}\centering
\includegraphics[scale=1]{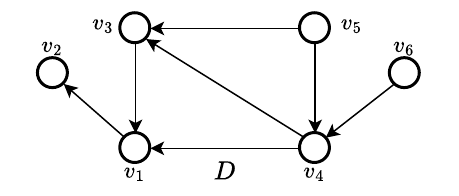}
\end{subfigure}
\begin{subfigure}[b]{0.49\linewidth}
\hspace{-3mm}\centering
\includegraphics[scale=1]{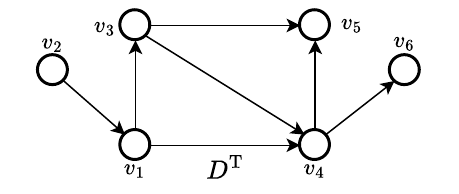}
\end{subfigure}
\caption{An example oriented graph that is WDGSS.}
\label{fig:example_wdgss}
\end{figure}
It can be checked that $D$ is not self-transpose. Hence, it follows the prediction by Corollary \ref{corollary:wdgss}, that $D$ is weakly determined by the generalized skew spectrum (WDGSS). Its only generalized cospectral mate, up to isomorphism, is its transpose $D^{\mathrm{T}}$. Moreover, let $Q \in \Gamma(D)$ such that $S((D)^{\mathrm{T}})= Q^{\mathrm{T}} S(D) Q$. We find that the level $\ell(Q)=191$, satisfying Lemma \ref{lemma:isomorphic}.
\end{example}

% \textcolor{red} {Raza: There is too much discussion on controllability. I think we should minimize it.}

% \todo[inline]{Waseem: I have modified the section. I am happy with the current text.}

In the next section, we explore the potential of network control theory in addressing the spectral characterization of graphs problem. 

\section{Control-Theoretic Perspectives on Spectral Characterization of Graphs}
\label{section:controls}
As discussed earlier, recent methods addressing which graphs are determined by spectrum (DS) problem and its generalized version (DGS) have employed the walk matrix. Interestingly, the walk matrix yields an intriguing interpretation in the context of dynamical systems on graphs (networked dynamical systems), as we now explain. Consider the linear dynamical system defined on a graph:

\begin{equation}
\label{eq:dyn}
    \dot{x}(t) = Ax(t) + e u(t)
\end{equation}
where \( {x}(t) \in \mathbb{R}^n\) is the state vector of the nodes; $A$ is the adjacency matrix encoding the coupling between nodes, \( u(t) \) is the control input, and \( e \) is the input vector specifying which node(s) receive control signals. A fundamental question in control theory is whether the system is \emph{controllable}---that is, whether one can drive the system from any initial state to any desired final state in finite time using appropriate inputs $u(t)$. Controllability is determined by the \emph{controllability matrix}:
$$
\mathcal{C} = \left[e, \; Ae, \; A^2 e, \; \cdots, \; A^{n-1}e \right],
$$
and the system \eqref{eq:dyn} is controllable if and only if $\det{\mathcal{C}}\ne 0$. Notably, $\mathcal{C}$ is exactly the walk matrix of the graph. Thus, the walk matrix-based criteria discussed earlier can be reinterpreted through the lens of system controllability. In other words, controllable graphs that also satisfy certain additional constraints (e.g., on the determinant of the controllability matrix) are DGS. This perspective builds a bridge between spectral graph theory and control theory, suggesting a new approach to the spectral characterization problem. While \eqref{eq:dyn} is a canonical example, it represents just one possible dynamical system on a graph. A more general formulation is:

\begin{equation}
\label{eq:dyna_2}
\dot{x}(t) = \mathcal{M}(G)x(t) + \mathcal{H}(G)u(t),
\end{equation}
where \( \mathcal{M}(G) \) is a structure-dependent system matrix (not necessarily the adjacency matrix), and \( \mathcal{H}(G) \) defines the control input configuration. The controllability matrix of \eqref{eq:dyna_2} is:
$$
\mathcal{C} = \left[ \mathcal{H}(G), \; \mathcal{M}(G)\mathcal{H}(G), \; \dots \;, (\mathcal{M}(G))^{n-1}\mathcal{H}(G) \right].
$$

This generalization raises a natural question: \emph{Can a more general controllability framework arising from flexible choices $\mathcal{M}(G)$ and $\mathcal{H}(G)$ offer new insights into the spectral characterization of graphs problem?} By exploring alternative system matrices such as the Laplacian, normalized Laplacian, or signless Laplacian, and by varying input structures via $\mathcal{H}(G)$, we may uncover new families of graphs that are determined by their spectra. This, in turn, could lead to a richer classification of DGS graphs and deepen the connection between spectral properties and control-theoretic structure.

\section{Concluding remarks}
\label{sec:conclusion}

In this work, we analyzed the spectral properties of oriented graphs to determine the maximum number of generalized cospectral mates they can possess. The central result, Theorem~\ref{theorem:main}, establishes an arithmetic criterion based on the determinant of the skew-walk matrix, enabling the classification of oriented graphs by their generalized skew spectrum. This criterion provides precise bounds on the number of generalized cospectral mates of a family of oriented graphs. As a special case, we derived conditions for identifying whether a graph is WDGSS.

A promising direction for future work is to relax the constraints in Theorem~\ref{theorem:main} to encompass a broader class of controllable oriented graphs. Preliminary observations indicate that the square-free condition on $\det{W(D)}$ may be further weakened. Moreover, extending these techniques to undirected graphs is an intriguing direction. We conjecture that an analogous upper bound on the number of generalized cospectral mates can be established in the undirected setting.

\end{document}